\documentclass[a4paper,11pt]{amsart}
\usepackage{amsmath,amssymb,amsfonts,latexsym}

\newtheorem{theorem}{Theorem}[section]

\newtheorem{corollary}[theorem]{Corollary}

\input{tcilatex}

\begin{document}
\title[\textbf{Two product Genocchi and Bernoulli polynomials}]{\textbf{Some
new identities of Genocchi numbers and polynomials involving Bernoulli and
Euler polynomials}}
\author[\textbf{S. Araci}]{\textbf{Serkan Araci}}
\address{\textbf{University of Gaziantep, Faculty of Science and Arts,
Department of Mathematics, 27310 Gaziantep, TURKEY}}
\email{\textbf{mtsrkn@hotmail.com, saraci88@yahoo.com.tr}}
\author[\textbf{M. Acikgoz}]{\textbf{Mehmet Acikgoz}}
\address{\textbf{University of Gaziantep, Faculty of Science and Arts,
Department of Mathematics, 27310 Gaziantep, TURKEY}}
\email{\textbf{acikgoz@gantep.edu.tr}}
\author[\textbf{E. \c{S}en}]{\textbf{Erdo\u{g}an \c{S}en}}
\address{\textbf{Department of Mathematics, Faculty of Science and Letters,
Nam\i k Kemal University, 59030 Tekirda\u{g}, TURKEY}}
\email{\textbf{erdogan.math@gmail.com}}

\begin{abstract}
In this paper, we deal with some new formulae for product of two Genocchi
polynomials together with both Euler polynomials and Bernoulli polynomials.
We get some applications for Genocchi polynomials. Our applications possess
a number of interesting properties to study in theory of Analytic numbers
which we express in the present paper.

\vspace{2mm}\noindent \textsc{2010 Mathematics Subject Classification.}
11S80, 11B68.

\vspace{2mm}

\end{abstract}

\maketitle




\section{\textbf{Introduction}}


The history of Genocchi numbers can be traced back to Italian mathematician
Angelo Genocchi (1817-1889). From Genocchi to the present time, Genocchi
numbers have been extensively studied in many different context in such
branches of Mathematics as, for instance, elementary number theory, complex
analytic number theory, Homotopy theory (stable Homotopy groups of spheres),
differential topology (differential structures on spheres), theory of
modular forms (Eisenstein series), $p$-adic analytic number theory ($p$-adic 
$L$-functions), quantum physics (quantum Groups). The works of Genocchi
numbers and their combinatorial relations have received much attention \cite%
{araci 3}, \cite{Araci 4}, \cite{araci 2}, \cite{Araci 5}, \cite{Araci 6}, 
\cite{Araci 7}, \cite{Jolany 1}, \cite{Jolany 2}, \cite{Kim 8}, \cite{Kim 10}%
. For showing the value of this type of numbers and polynomials, we list
some of their applications.

In the complex plane, the Genocchi numbers, named after Angelo Genocchi, are
a sequence of integers that defined by the exponential generating function:%
\begin{equation}
\frac{2t}{e^{t}+1}=e^{Gt}=\sum_{n=0}^{\infty }G_{n}\frac{t^{n}}{n!},\text{ }%
\left( \left\vert t\right\vert <\pi \right)   \label{Equation 1}
\end{equation}%
with the usual convention about replacing $G^{n}$ by $G_{n}$, is used. When
we multiply with $e^{xt}$ in the left hand side of the Eq. (\ref{Equation 1}%
), then we have%
\begin{equation}
\sum_{n=0}^{\infty }G_{n}\left( x\right) \frac{t^{n}}{n!}=\frac{2t}{e^{t}+1}%
e^{xt},\text{ }\left( \left\vert t\right\vert <\pi \right) 
\label{Equation 28}
\end{equation}%
where $G_{n}\left( x\right) $ called Genocchi polynomials. It follows from (%
\ref{Equation 28}) that $%
G_{1}=1,G_{2}=-1,G_{3}=0,G_{4}=1,G_{5}=0,G_{6}=-3,G_{7}=0,G_{8}=17,\cdots ,$
and $G_{2n+1}=0$ for $n\in 
\mathbb{N}
$ (for details, see \cite{araci 3}, \cite{Araci 4}, \cite{araci 2}, \cite%
{Araci 5}, \cite{Jolany 1}, \cite{Jolany 2}, \cite{Kim 8}, \cite{Kim 10}).

Differentiating both sides of (\ref{Equation 1}), with respect to $x$, then
we have the following:%
\begin{equation}
\frac{d}{dx}G_{n}\left( x\right) =nG_{n-1}\left( x\right) \text{ \textit{and}
}\deg G_{n+1}\left( x\right) =n.  \label{Equation 2}
\end{equation}

On account of (\ref{Equation 1}) and (\ref{Equation 2}), we can easily
derive the following:%
\begin{equation}
\int_{b}^{a}G_{n}\left( x\right) dx=\frac{G_{n+1}\left( a\right)
-G_{n+1}\left( b\right) }{n+1}\text{.}  \label{Equation 3}
\end{equation}

By (\ref{Equation 1}), we get%
\begin{equation}
G_{n}\left( x\right) =\sum_{k=0}^{n}\binom{n}{k}G_{k}x^{n-k}.
\label{Equation 4}
\end{equation}

Thanks to (\ref{Equation 3}) and (\ref{Equation 4}), we acquire the
following equation (\ref{Equation 5}):%
\begin{equation}
\int_{0}^{1}G_{n}\left( x\right) dx=-2\frac{G_{n+1}}{n+1}\text{.}
\label{Equation 5}
\end{equation}

It is not difficult to see that%
\begin{eqnarray}
e^{tx} &=&\frac{1}{2t}\left( \frac{2t}{e^{t}+1}e^{\left( 1+x\right) t}+\frac{%
2t}{e^{t}+1}e^{xt}\right)  \label{Equation 6} \\
&=&\frac{1}{2t}\sum_{n=0}\left( G_{n}\left( x+1\right) +G_{n}\left( x\right)
\right) \frac{t^{n}}{n!}\text{.}  \notag
\end{eqnarray}

By expression of (\ref{Equation 6}), then we have%
\begin{equation}
2x^{n}=\frac{G_{n+1}\left( x+1\right) +G_{n+1}\left( x\right) }{n+1}
\label{Equation 7}
\end{equation}%
(see \cite{Kim 8}, \cite{Jolany 2}).

Let $\mathcal{P}_{n}\mathcal{=}\left\{ p\left( x\right) \in 
\mathbb{Q}
\left[ x\right] \mid \deg p\left( x\right) \leq n\right\} $ be the $\left(
n+1\right) $-dimensional vector space over $%
\mathbb{Q}
$. Probably, $\left\{ 1,x,x^{2},\cdots ,x^{n}\right\} $ is the most natural
basis for $\mathcal{P}_{n}$. From this, we note that $\left\{ G_{1}\left(
x\right) ,G_{2}\left( x\right) ,\cdots ,G_{n+1}\left( x\right) \right\} $ is
also good basis for space $\mathcal{P}_{n}$.

In \cite{Kim 5}, Kim $et$ $al$. introduced the following integrals:%
\begin{equation}
I_{m,n}=\int_{0}^{1}B_{m}\left( x\right) x^{n}dx\text{ \ and }%
J_{m,n}=\int_{0}^{1}E_{m}\left( x\right) x^{n}dx  \label{Equation 8}
\end{equation}%
where $B_{m}\left( x\right) $ and $E_{n}\left( x\right) $ are called
Bernoulli polynomials and Euler polynomials, respectively. Also, they are
defined by the following generating functions:%
\begin{eqnarray}
e^{B\left( x\right) t} &=&\sum_{n=0}^{\infty }B_{n}\left( x\right) \frac{%
t^{n}}{n!}=\frac{t}{e^{t}-1}e^{xt},\text{ }\left\vert t\right\vert <2\pi ,
\label{Equation 9} \\
e^{E\left( x\right) t} &=&\sum_{n=0}^{\infty }E_{n}\left( x\right) \frac{%
t^{n}}{n!}=\frac{2}{e^{t}+1}e^{xt},\text{ }\left\vert t\right\vert <\pi
\label{Equation 11}
\end{eqnarray}%
with $B^{n}\left( x\right) :=B_{n}\left( x\right) $ and $E^{n}\left(
x\right) :=E_{n}\left( x\right) $, symbolically. By substituting $x=0$ in (%
\ref{Equation 9}) and (\ref{Equation 11}), then we readily see that, 
\begin{eqnarray}
\frac{t}{e^{t}-1} &=&\sum_{n=0}^{\infty }B_{n}\left( 0\right) \frac{t^{n}}{n!%
},  \label{Equation 10} \\
\frac{2}{e^{t}+1} &=&\sum_{n=0}^{\infty }E_{n}\left( 0\right) \frac{t^{n}}{n!%
}.  \label{Equation 12}
\end{eqnarray}

Here $B_{n}\left( 0\right) :=B_{n}$ and $E_{n}\left( 0\right) :=E_{n}$ are
called Bernoulli numbers and Euler numbers, respectively. Thus, Bernoulli
and Euler numbers and polynomials have the following identities:%
\begin{equation}
B_{n}\left( x\right) =\sum_{k=0}^{n}\binom{n}{k}B_{k}x^{n-k}\text{ and }%
E_{n}\left( x\right) =\sum_{k=0}^{n}\binom{n}{k}E_{k}x^{n-k}.
\label{Equation 13}
\end{equation}%
(for details, see \cite{Acikgoz}, \cite{Bayad}, \cite{araci 1}, \cite{Cangul}%
, \cite{Kim 11}, \cite{Kim 9}, \cite{Luo}). By (\ref{Equation 10}) and (\ref%
{Equation 12}), we have the following recurrence relations of Euler and
Bernoulli numbers, as follows:%
\begin{equation}
B_{0}=1,\text{ }B_{n}\left( 1\right) -B_{n}=\delta _{1,n}\text{ and }E_{0}=1,%
\text{ }E_{n}\left( 1\right) +E_{n}=2\delta _{0,n}  \label{Equation 14}
\end{equation}%
where $\delta _{n,m}$ is the Kronecker's symbol which is defined by%
\begin{equation}
\delta _{n,m}=\left\{ 
\begin{array}{cc}
1, & \text{if }n=m \\ 
0, & \text{if }n\neq m.%
\end{array}%
\right.   \label{Equation 15}
\end{equation}

In the complex plane, we can write the following:%
\begin{equation}
\sum_{n=0}^{\infty }G_{n}\frac{\left( it\right) ^{n}}{n!}=it\frac{2}{e^{it}+1%
}=it\sum_{n=0}^{\infty }E_{n}\frac{\left( it\right) ^{n}}{n!}\text{.}
\label{Equation 37}
\end{equation}

By (\ref{Equation 37}), we have%
\begin{equation*}
\sum_{n=0}^{\infty }\left( \frac{G_{n+1}}{n+1}\right) \frac{\left( it\right)
^{n}}{n!}=\sum_{n=0}^{\infty }E_{n}\frac{\left( it\right) ^{n}}{n!},
\end{equation*}%
by comparing coefficients on the both sides of the above equatlity, then we
have%
\begin{equation}
\frac{G_{n+1}}{n+1}=E_{n},\text{ (see \cite{Kimm 8}).}  \label{Equation 29}
\end{equation}

Via the equation (\ref{Equation 29}), our results in the present paper can
be extended to Euler polynomials.

From of Eqs (\ref{Equation 8}-\ref{Equation 15}), Kim $et$ $al$. derived
some new formulae on the product for two and several Bernoulli and Euler
polynomials (for details, see [21-26]). In \cite{He}, He and Wang also gave
formulae of products of the Apostol-Bernoulli and Apostol-Euler Polynomials.

By the same motivation of the above knowledge, we write this paper. We give
some interesting properties which are procured from the basis of Genocchi.
From our methods, we obtain some new identities including Bernoulli and
Euler polynomials. Also, by using (\ref{Equation 29}), we derive our results
in terms of Euler polynomials.

\section{\textbf{On the Genocchi numbers and polynomials}}

In this section, we introduce the following integral equation: For $m,n\geq
1,$%
\begin{equation}
T_{m,n}=\int_{0}^{1}G_{m}\left( x\right) x^{n}dx\text{.}  \label{Equation 16}
\end{equation}

By (\ref{Equation 16}), becomes:%
\begin{equation*}
T_{m,n}=-\frac{G_{m+1}}{m+1}-\frac{n}{m+1}\int_{0}^{1}G_{m+1}\left( x\right)
x^{n-1}dx\text{.}
\end{equation*}

Thus, we have the following recurrence formulas, as follows:%
\begin{equation*}
T_{m,n}=-\frac{G_{m+1}}{m+1}-\frac{n}{m+1}T_{m+1,n-1}
\end{equation*}%
by continuing with the above recurrence relation, then we derive that%
\begin{equation*}
T_{m,n}=-\frac{G_{m+1}}{m+1}+\left( -1\right) ^{2}\frac{n}{\left( m+1\right)
\left( m+2\right) }G_{m+2}+\left( -1\right) ^{2}\frac{n\left( n-1\right) }{%
\left( m+1\right) \left( m+2\right) }T_{m+2,n-2}.
\end{equation*}

Now also, we develop the following for sequel of this paper:%
\begin{equation}
T_{m,n}=\frac{1}{n+1}\sum_{j=1}^{n}\left( -1\right) ^{j}\frac{\binom{n+1}{j}%
}{\binom{m+j}{m}}G_{m+j}+2\frac{\left( -1\right) ^{n+1}G_{n+m+1}}{\left(
n+m+1\right) \binom{n+m}{m}}.  \label{Equation 17}
\end{equation}

Let us now introduce the polynomial 
\begin{equation*}
p\left( x\right) =\sum_{l=0}^{n}G_{l}\left( x\right) x^{n-l}\text{, with }%
n\in 
\mathbb{N}
.
\end{equation*}

Taking $k$-th derivative of the above equality, then we have 
\begin{eqnarray}
p^{\left( k\right) }\left( x\right) &=&\left( n+1\right) n\left( n-1\right)
\cdots \left( n-k+2\right) \sum_{l=k}^{n}G_{l-k}\left( x\right) x^{n-l}
\label{Equation 18} \\
&=&\frac{\left( n+1\right) !}{\left( n-k+1\right) !}\sum_{l=k}^{n}G_{l-k}%
\left( x\right) x^{n-l}\text{ }\left( k=0,1,2,\cdots ,n\right) .  \notag
\end{eqnarray}

\begin{theorem}
\label{Theorema}The following equality holds true:%
\begin{gather*}
\sum_{l=0}^{n}G_{l}\left( x\right) x^{n-l} \\
=\sum_{k=1}^{n-1}\left( \sum_{j=1}^{n-k}\left( -1\right) ^{j}\frac{\binom{%
n-k+1}{j}}{\left( n-k+1\right) \binom{k+j}{k}}G_{k+j}+2\frac{\left(
-1\right) ^{n-k+1}G_{n+1}}{\left( n+1\right) \binom{n}{k}}-2\frac{G_{k+1}}{%
k+1}\right) \\
+\sum_{k=1}^{n}\left( \frac{\binom{n+2}{k}}{n+2}\sum_{l=k-1}^{n-1}\left(
2-G_{l-k+1}-G_{n-k+1}\right) \right) B_{k}\left( x\right) \text{.}
\end{gather*}
\end{theorem}

\begin{proof}
On account of the properties of the Genocchi basis for the space of
polynomials of degree less than or equal to $n$ with coefficients in $%
\mathbb{Q}
$, then $p\left( x\right) $ can be written as follows:%
\begin{equation}
p\left( x\right) =\sum_{k=0}^{n}a_{k}B_{k}\left( x\right)
=a_{0}+\sum_{k=1}^{n}a_{k}B_{k}\left( x\right) \text{.}  \label{Equation 19}
\end{equation}%
Therefore, by (\ref{Equation 19}), we obtain%
\begin{align}
a_{0}& =\int_{0}^{1}p\left( x\right)
dx=\sum_{k=1}^{n}\int_{0}^{1}G_{k}\left( x\right)
x^{n-k}dx=\sum_{k=1}^{n}T_{k,n-k}=\sum_{k=1}^{n-1}T_{k,n-k}+T_{k,0}
\label{Equation 30} \\
& =\sum_{k=1}^{n-1}\frac{1}{n-k+1}\sum_{j=1}^{n-k}\left( -1\right) ^{j}\frac{%
\binom{n-k+1}{j}}{\binom{k+j}{k}}G_{k+j}+2\frac{\left( -1\right)
^{n-k+1}G_{n+1}}{\left( n+1\right) \binom{n}{k}}-2\frac{G_{k+1}}{k+1}\text{.}
\notag
\end{align}%
From expression of (\ref{Equation 18}), we get%
\begin{eqnarray}
a_{k} &=&\frac{1}{k!}\left( p^{\left( k-1\right) }\left( 1\right) -p^{\left(
k-1\right) }\left( 0\right) \right)  \label{Equation 31} \\
&=&\frac{\left( n+1\right) !}{k!\left( n-k+2\right) !}\left(
\sum_{l=k-1}^{n}G_{l-k+1}\left( 1\right) -0^{n-l}G_{n-k+1}\right)  \notag \\
&=&\frac{\binom{n+2}{k}}{n+2}\sum_{l=k-1}^{n-1}\left(
2-G_{l-k+1}-G_{n-k+1}\right) \text{.}  \notag
\end{eqnarray}

Substituting equations (\ref{Equation 30}) and (\ref{Equation 31}) into (\ref%
{Equation 19}), we arrive at the desired result.
\end{proof}

By using (\ref{Equation 29}) and theorem \ref{Theorema}, we get the
following corollary, which has been stated in terms of Euler polynomials.

\begin{corollary}
For any $n\in 
\mathbb{N}
$, then we have%
\begin{gather*}
\sum_{l=0}^{n}G_{l}\left( x\right) x^{n-l} \\
=\sum_{k=1}^{n-1}\left( \sum_{j=1}^{n-k}\left( -1\right) ^{j}\frac{\left(
k+j\right) \binom{n-k+1}{j}}{\left( n-k+1\right) \binom{k+j}{j}}E_{k+j-1}+2%
\frac{\left( -1\right) ^{n-k+1}E_{n}}{\binom{n}{k}}-2E_{k}\right) \\
+\sum_{k=1}^{n}\left( \frac{\binom{n+2}{k}}{n+2}\sum_{l=k-1}^{n-1}\left(
2-\left( l-k+1\right) E_{l-k}-\left( n-k+1\right) E_{n-k}\right) \right)
B_{k}\left( x\right) \text{.}
\end{gather*}
\end{corollary}

\begin{theorem}
\label{Theorem 2}The following nice identity%
\begin{gather*}
\sum_{l=0}^{n}G_{l}\left( x\right) x^{n-l} \\
=\sum_{k=0}^{n}\left( \left( n+1\right) \binom{n}{k}-\frac{\binom{n+1}{k}}{2}%
\sum_{l=k}^{n-1}\left( G_{l-k}-G_{n-k}\right) \right) E_{k}\left( x\right)
\end{gather*}%
is true.
\end{theorem}

\begin{proof}
Let us now consider the polynomial $p\left( x\right) $ in terms of Euler
polynomials as follows:%
\begin{equation*}
p\left( x\right) =\sum_{k=0}^{n}b_{k}E_{k}\left( x\right) \text{.}
\end{equation*}%
In \cite{Kim 5}, Kim $et$ $al$. gave the coefficients $b_{k}$ by utilizing
from the definition of Bernoulli polynomials. Now also, we give the
coefficients $b_{k}$ by using the definition of Genocchi polynomials, as
follows:%
\begin{eqnarray*}
b_{k} &=&\frac{1}{2k!}\left( p^{\left( k\right) }\left( 1\right) +p^{\left(
k\right) }\left( 0\right) \right) \\
&=&\frac{\left( n+1\right) !}{2k!\left( n-k+1\right) !}\sum_{l=k}^{n}\left(
G_{l-k}\left( 1\right) +0^{n-l}G_{l-k}\right) \\
&=&\left( n+1\right) \binom{n}{k}-\frac{\binom{n+1}{k}}{2}%
\sum_{l=k}^{n-1}\left( G_{l-k}-G_{n-k}\right) \text{.}
\end{eqnarray*}%
After the above applications, we complete the proof of theorem.
\end{proof}

By employing (\ref{Equation 29}) and theorem \ref{Theorem 2}, we have the
following corollary, which is sums of products of two Euler polynomials.

\begin{corollary}
For each $n\in 
\mathbb{N}
$, then we have%
\begin{gather*}
\sum_{l=0}^{n}G_{l}\left( x\right) x^{n-l} \\
=\sum_{k=0}^{n}\left( \left( n+1\right) \binom{n}{k}-\frac{\binom{n+1}{k}}{2}%
\sum_{l=k}^{n-1}\left( \left( l-k\right) E_{l-k-1}-\left( n-k\right)
E_{n-k-1}\right) \right) E_{k}\left( x\right) \text{.}
\end{gather*}
\end{corollary}

We now discover the following theorem, which will be interesting and
worthwhile theorem for studying in Analytic numbers theory.

\begin{theorem}
The following equality holds:%
\begin{gather*}
\sum_{l=0}^{n}\frac{1}{l!\left( n-l\right) !}G_{l}\left( x\right) x^{n-l} \\
=\sum_{l=1}^{n}\frac{2^{l-2}}{l!}\sum_{j=l-1}^{n}\frac{\left(
2-G_{l-j+1}\right) G_{l}\left( x\right) }{\left( j-l+1\right) !\left(
n-j\right) !}+\frac{2^{l-2}}{l!\left( n-l+1\right) !}G_{n-l+1}G_{l}\left(
x\right) \text{.}
\end{gather*}
\end{theorem}

\begin{proof}
It is proved by using the following polynomial $p\left( x\right) :$ 
\begin{equation}
p\left( x\right) =\sum_{l=0}^{n}\frac{1}{l!\left( n-l\right) !}G_{l}\left(
x\right) x^{n-l}=\sum_{l=0}^{n}a_{l}G_{l}\left( x\right) \text{.}
\label{Equation 33}
\end{equation}%
It is not difficult to indicate the following:%
\begin{equation}
p^{\left( k\right) }\left( x\right) =2^{k}\sum_{l=k}^{n}\frac{1}{\left(
l-k\right) !\left( n-l\right) !}G_{l-k}\left( x\right) x^{n-l}\text{.}
\label{Equation 20}
\end{equation}%
Then, we see that for $k=1,2,\cdots ,n,$%
\begin{align}
a_{l}& =\frac{1}{2l!}\left( p^{\left( l-1\right) }\left( 1\right) +p^{\left(
l-1\right) }\left( 0\right) \right)  \label{Equation 32} \\
& =\frac{2^{l-2}}{l!}\sum_{j=l-1}^{n}\frac{1}{\left( j-l+1\right) !\left(
n-j\right) !}\left( G_{j-l+1}\left( 1\right) +0^{n-j}G_{j-l+1}\right)  \notag
\\
& =\frac{2^{l-2}}{l!}\sum_{j=l-1}^{n}\frac{\left( 2-G_{l-j+1}\right) }{%
\left( j-l+1\right) !\left( n-j\right) !}+\frac{2^{l-2}}{l!\left(
n-l+1\right) !}G_{n-l+1}.  \notag
\end{align}

By (\ref{Equation 33}) and (\ref{Equation 32}), we arrive at the desired
result.
\end{proof}

\begin{theorem}
\label{theorem 3}The following identity%
\begin{gather}
\sum_{l=0}^{n}\frac{1}{l!\left( n-l\right) !}G_{l}\left( x\right) x^{n-l}
\label{Equation 23} \\
=-2\frac{G_{n+1}}{n+1}+\sum_{l=1}^{n-1}\sum_{j=1}^{n-l}\frac{\left(
-1\right) ^{j}}{l!\left( n-l+1\right) !}\frac{\binom{n-l+1}{j}}{\binom{l+j}{l%
}}G_{l+j}+2\frac{\left( -1\right) ^{n-l+1}G_{n+1}}{\left( n+1\right) \binom{n%
}{l}}  \notag \\
+\sum_{k=1}^{n}\left( \frac{2^{k-1}}{k!}\sum_{l=k-1}^{n}\frac{\left(
2-G_{l-k+1}\right) }{\left( l-k+1\right) !\left( n-l\right) !}-\frac{2^{k-1}%
}{k!\left( n-k+1\right) !}G_{n-k+1}\right) B_{k}\left( x\right)  \notag
\end{gather}%
is true.
\end{theorem}

\begin{proof}
Now also, let us take the polynomial in terms of Bernoulli polynomials as 
\begin{equation}
p\left( x\right) =\sum_{k=0}^{n}a_{k}B_{k}\left( x\right) .
\label{Equation 34}
\end{equation}%
By using the above identity, we develop as follows:%
\begin{align}
a_{0}& =\int_{0}^{1}p\left( x\right) dx=\sum_{l=0}^{n}\frac{1}{l!\left(
n-l\right) !}\int_{0}^{1}G_{l}\left( x\right) x^{n-l}dx  \label{Equation 35}
\\
& =\sum_{l=0}^{n}\frac{1}{l!\left( n-l\right) !}T_{l,n-l}=T_{n,0}+%
\sum_{l=1}^{n-1}\frac{1}{l!\left( n-l\right) !}T_{l,n-l}  \notag \\
& =-2\frac{G_{n+1}}{n+1}+\sum_{l=1}^{n-1}\sum_{j=1}^{n-l}\frac{\left(
-1\right) ^{j}}{l!\left( n-l+1\right) !}\frac{\binom{n-l+1}{j}}{\binom{l+j}{l%
}}G_{l+j}+2\frac{\left( -1\right) ^{n-l+1}G_{n+1}}{\left( n+1\right) \binom{n%
}{l}}.  \notag
\end{align}%
By (\ref{Equation 20}), we compute $a_{k}$ coefficients, as follows:%
\begin{eqnarray}
a_{k} &=&\frac{1}{k!}\left( p^{\left( k-1\right) }\left( 1\right) -p^{\left(
k-1\right) }\left( 0\right) \right)  \label{Equation 36} \\
&=&\frac{2^{k-1}}{k!}\sum_{l=k-1}^{n}\frac{1}{\left( l-k+1\right) !\left(
n-l\right) !}\left( G_{l-k+1}\left( 1\right) -0^{n-l}G_{l-k+1}\right)  \notag
\\
&=&\frac{2^{k-1}}{k!}\sum_{l=k-1}^{n}\frac{\left( 2-G_{l-k+1}\right) }{%
\left( l-k+1\right) !\left( n-l\right) !}-\frac{2^{k-1}}{k!\left(
n-k+1\right) !}G_{n-k+1}\text{.}  \notag
\end{eqnarray}

When we substituted (\ref{Equation 35}) and (\ref{Equation 36}) into (\ref%
{Equation 34}), the proof of theorem will be completed.
\end{proof}

By using equation (\ref{Equation 29}) and theorem \ref{theorem 3}, we
procure the following corollary.

\begin{corollary}
For any $n\in 
\mathbb{N}
,$ then we have%
\begin{gather*}
\sum_{l=0}^{n}\frac{1}{l!\left( n-l\right) !}G_{l}\left( x\right) x^{n-l} \\
=-2E_{n}+\sum_{l=1}^{n-1}\sum_{j=1}^{n-l}\frac{\left( -1\right) ^{j}}{%
l!\left( n-l+1\right) !}\frac{\left( l+j\right) \binom{n-l+1}{j}}{\binom{l+j%
}{l}}E_{l+j-1}+2\frac{\left( -1\right) ^{n-l+1}E_{n}}{\binom{n}{l}} \\
+\sum_{k=1}^{n}\left( \frac{2^{k-1}}{k!}\sum_{l=k-1}^{n}\frac{\left( \frac{2%
}{l-k+1}-E_{l-k}\right) }{\left( l-k\right) !\left( n-l\right) !}-\frac{%
2^{k-1}}{k!\left( n-k\right) !}E_{n-k}\right) B_{k}\left( x\right)
\end{gather*}
\end{corollary}

In \cite{Kim 8}, it is well-known that%
\begin{equation}
G_{n}\left( x+y\right) =\sum_{k=0}^{n}\binom{n}{k}G_{k}\left( x\right)
y^{n-k}\text{.}  \label{Equation 21}
\end{equation}

For $x=y$ in (\ref{Equation 21}), then we have the following%
\begin{equation}
\frac{1}{n!}G_{n}\left( 2x\right) =\sum_{k=0}^{n}\frac{1}{k!\left(
n-k\right) !}G_{k}\left( x\right) x^{n-k}.  \label{Equation 22}
\end{equation}

By comparing the equations of (\ref{Equation 23}) and (\ref{Equation 22}),
then we readily derive the following corollary.

\begin{corollary}
\begin{equation*}
\frac{1}{n!}G_{n}\left( 2x\right) =\text{the right-hand-side of equation in
Theorem 2.4.}
\end{equation*}
\end{corollary}

\begin{theorem}
\label{theorem 4}The following equality%
\begin{gather*}
\sum_{k=1}^{n-1}\frac{1}{k\left( n-k\right) }G_{k}\left( x\right) x^{n-k} \\
=\sum_{k=0}^{n}\left( \frac{\binom{n}{k}}{2\left( n-k+1\right) }\left(
H_{n-1}-H_{n-k}\right) -\frac{\binom{n}{k}}{2n}\sum_{l=k}^{n-1}\frac{\left(
2-G_{l-k+1}\right) }{\left( n-l\right) \left( l-k+1\right) }\right)
G_{k}\left( x\right)
\end{gather*}%
holds true.
\end{theorem}

\begin{proof}
To prove this theorem, we introduce the following polynomial $p\left(
x\right) :$ 
\begin{equation*}
p\left( x\right) =\sum_{k=1}^{n-1}\frac{1}{k\left( n-k\right) }G_{k}\left(
x\right) x^{n-k}\text{.}
\end{equation*}%
Then, we derive $k$-th derivative of $p\left( x\right) $ is given by%
\begin{equation}
p^{\left( k\right) }\left( x\right) =C_{k}\left( x^{n-k}+G_{n-k}\left(
x\right) \right) +\left( n-1\right) \left( n-2\right) \cdots \left(
n-k\right) \sum_{l=k+1}^{n-1}\frac{G_{l-k}\left( x\right) x^{n-l}}{\left(
n-l\right) \left( l-k\right) },  \label{Equation 24}
\end{equation}%
where%
\begin{equation*}
C_{k}=\frac{\sum_{j=1}^{k}\left( n-1\right) ...\left( n-j+1\right) \left(
n-j-1\right) ...\left( n-k\right) }{n-k}\text{ }\left( k=1,2,...,n-1\right) 
\text{, }C_{0}=0\text{.}
\end{equation*}%
We want to note that%
\begin{equation*}
p^{\left( n\right) }\left( x\right) =\left( p^{\left( n-1\right) }\left(
x\right) \right) 
{\acute{}}%
=C_{n-1}\left( x+G_{1}\left( x\right) \right) =C_{n-1}=\left( n-1\right)
!H_{n-1},
\end{equation*}%
where $H_{n-1}$ are called Harmonic numbers, which are defined by%
\begin{equation*}
H_{n-1}=\sum_{j=1}^{n-1}\frac{1}{j}\text{.}
\end{equation*}%
With the properties of Genocchi basis for the space of polynomials of degree
less than or equal to $n$ with coefficients in $%
\mathbb{Q}
$, $p\left( x\right) $ is introduced by%
\begin{equation}
p\left( x\right) =\sum_{k=0}^{n}a_{k}G_{k}\left( x\right) \text{.}
\label{Equation 25}
\end{equation}%
By expression of (\ref{Equation 25}), we obtain that%
\begin{eqnarray*}
a_{k} &=&\frac{1}{2k!}\left( p^{\left( k-1\right) }\left( 1\right)
+p^{\left( k-1\right) }\left( 0\right) \right) \\
&=&\frac{C_{k-1}}{2k!}\left( 1+2\delta _{1,n-k+1}\right) +\frac{\left(
n-1\right) !}{2k!\left( n-k\right) !}\sum_{l=k}^{n-1}\frac{\left(
G_{l-k+1}\left( 1\right) +0^{n-l}G_{l-k+1}\right) }{\left( n-l\right) \left(
l-k+1\right) } \\
&=&\frac{C_{k-1}}{2k!}-\frac{\binom{n}{k}}{2n}\sum_{l=k}^{n-1}\frac{\left(
2-G_{l-k+1}\right) }{\left( n-l\right) \left( l-k+1\right) }.
\end{eqnarray*}%
As a result,%
\begin{equation*}
a_{n}=\frac{1}{2n!}\left( p^{\left( n\right) }\left( 1\right) +p^{\left(
n\right) }\left( 0\right) \right) =\frac{C_{n-1}}{n!}=\frac{H_{n-1}}{n}\text{%
.}
\end{equation*}%
In \cite{Kim 5}, it is well-known that%
\begin{equation}
\frac{C_{k-1}}{k!}=\frac{\binom{n}{k}}{\left( n-k+1\right) }\left(
H_{n-1}-H_{n-k}\right) \text{.}  \label{Equation 26}
\end{equation}%
By (\ref{Equation 24}), (\ref{Equation 25}) and (\ref{Equation 26}), then we
arrive at the desired result.
\end{proof}

From (\ref{Equation 29}) and (\ref{theorem 4}), we acquire the following.

\begin{corollary}
The following identity holds:%
\begin{gather*}
\sum_{k=1}^{n-1}\frac{1}{k\left( n-k\right) }G_{k}\left( x\right) x^{n-k} \\
=\sum_{k=1}^{n}\left( \frac{\binom{n}{k}}{2\left( n-k+1\right) }\left(
H_{n-1}-H_{n-k}\right) -\frac{\binom{n}{k}}{2n}\sum_{l=k}^{n-1}\frac{\left( 
\frac{2}{l-k+1}-E_{l-k}\right) }{\left( n-l\right) }\right) kE_{k-1}\left(
x\right)
\end{gather*}
\end{corollary}

\section{\textbf{Further Remarks}}

Let $\mathcal{P}_{n}=\left\{ \sum_{j=0}a_{j}x^{j}\mid a_{j}\in 
\mathbb{Q}
\right\} $ be the space of polynomials of degree less than or equal to $n.$
In this final section, we will give the matrix formulation of Genocchi
polynomials. Let us now consider the polynomial $p\left( x\right) \in 
\mathcal{P}_{n}$ as a linear combination of Genocchi basis polynomials with%
\begin{equation*}
p\left( x\right) =C_{1}G_{1}\left( x\right) +C_{2}G_{2}\left( x\right)
+\cdots +C_{n+1}G_{n+1}\left( x\right) \text{.}
\end{equation*}

We can write the above as a product of two variables%
\begin{equation}
p\left( x\right) =\left( 
\begin{array}{cccc}
G_{1}\left( x\right) & G_{2}\left( x\right) & \cdots & G_{n+1}\left( x\right)%
\end{array}%
\right) \left( 
\begin{array}{c}
C_{1} \\ 
C_{2} \\ 
\vdots \\ 
C_{n+1}%
\end{array}%
\right) .  \label{Equation 27}
\end{equation}

From expression of (\ref{Equation 27}), we consider the following equation:%
\begin{equation*}
p\left( x\right) =\left( 
\begin{array}{ccccc}
1 & x & x^{2} & \cdots & x^{n}%
\end{array}%
\right) \left( 
\begin{array}{cccc}
g_{1,1} & g_{1,2} & \cdots & g_{1,n+1} \\ 
0 & g_{2,2} & \cdots & g_{2,n+1} \\ 
0 & 0 & \cdots & g_{3,n+1} \\ 
\vdots & \vdots & \ddots & \vdots \\ 
0 & 0 & 0 & g_{n+1,n+1}%
\end{array}%
\right) \left( 
\begin{array}{c}
C_{1} \\ 
C_{2} \\ 
C_{3} \\ 
\vdots \\ 
C_{n+1}%
\end{array}%
\right)
\end{equation*}%
where $g_{i,j}$ are the coefficients of the power basis that are used to
determine the respective Genocchi polynomials. We now list a few Genocchi
polynomials as follows:%
\begin{equation*}
G_{1}\left( x\right) =1,\text{ }G_{2}\left( x\right) =2x-1,\text{ }%
G_{3}\left( x\right) =3x^{2}-3x,\text{ }G_{4}\left( x\right)
=4x^{3}-6x^{2}-1,\cdots .
\end{equation*}

In the quadratic case ($n=2$), the matrix representation is%
\begin{equation*}
p\left( x\right) =\left( 
\begin{array}{ccc}
1 & x & x^{2}%
\end{array}%
\right) \left( 
\begin{array}{ccc}
1 & -1 & 0 \\ 
0 & 2 & -3 \\ 
0 & 0 & 3%
\end{array}%
\right) \left( 
\begin{array}{c}
C_{1} \\ 
C_{2} \\ 
C_{3}%
\end{array}%
\right) \text{.}
\end{equation*}

In the cubic case ($n=3$), the matrix representation is%
\begin{equation*}
p\left( x\right) =\left( 
\begin{array}{cccc}
1 & x & x^{2} & x^{3}%
\end{array}%
\right) \left( 
\begin{array}{cccc}
1 & -1 & 0 & -1 \\ 
0 & 2 & -3 & 0 \\ 
0 & 0 & 3 & -6 \\ 
0 & 0 & 0 & 4%
\end{array}%
\right) \text{.}
\end{equation*}

Throughout this paper, many considerations for Genocchi polynomials seem to
be useful for a matrix formulation.\hspace{0.5cm}

%

\end{document}